\documentclass[11pt]{amsart}
\usepackage[margin=1in]{geometry}
\usepackage{amsmath, amsfonts, amssymb, amsthm}
\usepackage{amsrefs}
\usepackage{mathrsfs}
\usepackage{enumitem} 
\usepackage{hyperref}
\usepackage{color}
\usepackage{dsfont}
\usepackage{csquotes}
\usepackage{epigraph}
\usepackage{comment}
\usepackage{mlmodern}
\usepackage[T1]{fontenc}
\usepackage{color}
\definecolor{darkblue}{rgb}{0.0,0.0,0.3}
\hypersetup{
    colorlinks=false,
    linkcolor=blue,
    urlcolor=darkblue,
    }

\newtheorem{theorem}{Theorem}[section]

\newtheorem{lemma}[theorem]{Lemma}

\theoremstyle{definition}

\newtheorem{definition}[theorem]{Definition}

\newtheorem{notation}[theorem]{Notation}

\theoremstyle{remark}
\newtheorem{remark}[theorem]{Remark}

\numberwithin{equation}{section}

\newcommand{\bR}{\mathbb{R}}

\newcommand{\bS}{\mathbb{S}}

\newcommand{\Ampere}{Amp\`{e}re}

\newcommand{\Garding}{G\r{a}rding}
\newcommand\norm[1]{\left\lVert#1\right\rVert}

\title[A Class of Hessian Type Equations]{Starshaped Compact Hypersurfaces in Warped Product Manifolds II: A Class of Hessian Type Equations}
\author{Bin Wang}
\address[]{Department of Mathematics, The Chinese University of Hong Kong, the Hong Kong Special Administrative Region, People's Republic of China.} 
\email{bwang@math.cuhk.edu.hk}
\subjclass[2020]{Primary 53C21; Secondary 35J60}
\keywords{Prescribed Weingarten curvature equations, Hessian type equations, second order estimates.}

\begin{document}

\begin{abstract}
In this note, we prove the existence of one particular class of starshaped compact hypersurfaces, by deriving global curvature estimates for such hypersurfaces; this generalizes the main result in [Hypersurfaces of prescribed mixed Weingarten curvature. J. Geom. Anal. \textbf{34} (2024).] from the Euclidean space to Riemannian warped products. Moreover, we show that interior second order a priori estimates for admissible solutions to the associated fully nonlinear elliptic partial differential equations can be readily established by similar arguments.
\end{abstract}
\maketitle
\setcounter{tocdepth}{1}
\tableofcontents

\section{Introduction}

Let $\sigma_k: \bR^n \to \bR$ denote the $k$-th elementary symmetric polynomial
\[\sigma_k(x_1,\ldots,x_n)=\sum_{1 \leq i_1<\cdots<i_k \leq n}x_{i_1}\cdots x_{i_k}.\] For a symmetric matrix $A$, we denote by $\lambda(A)=(\lambda_1,\ldots,\lambda_n)$ its eigenvalues and the notation $\sigma_{k}(A)$ shall be interpreted as the quantity $\sigma_k(\lambda_1,\ldots,\lambda_n)$. When $A=D^2u$ is the Hessian matrix of some $u \in C^2(\Omega) \cap C(\overline{\Omega})$, then $\sigma_1(D^2u)=\Delta u$ is the Laplacian and $\sigma_n(D^2u)=\det(D^2u)$ is the Monge-\Ampere\ operator. When $A=(h_{ij})$ is the second fundamental form of a hypersurface $\Sigma$, then its eigenvalues are the principal curvatures $\kappa[\Sigma]=(\kappa_1,\ldots,\kappa_n)$; in this case, $\sigma_1(\kappa[\Sigma])=\sum_{i=1}^{n} \kappa_i$ is the mean curvature, $\sigma_2(\kappa[\Sigma])=\sum_{i<j}\kappa_i\kappa_j$ is the scalar curvature, and $\sigma_n(\kappa[\Sigma])=\kappa_1\cdots\kappa_n$ is the Gauss curvature. In general, the quantities $\sigma_k(D^2u)$ and $\sigma_k(\kappa[\Sigma])$ are called the $k$-Hessian and the $k$-curvature, respectively.

Ever since the seminal work \cite{CNS-3,CNS-4,CNS-5} of Caffarelli-Nirenberg-Spruck, there has been a vast literature studying partial differential equations which involve the $\sigma_k$ operators. Indeed, besides the above simple examples, many problems in analysis and geometry often arise as some type of the $\sigma_k$ equation i.e. equations of the form $\sigma_k=f$, or the $(k,l)$-quotient equations $\sigma_k/\sigma_l=f$. Some notable recent examples include the problem of prescribing $k$-th curvature measure \cite{Guan-Li-Li, Yang}, which is important in the field of convex geometry; and variants of the Minkowski problem, see \cite{Guan-Guan}.

\subsection{Literature review and the main result}
\indent

In this note, we are concerned with the existence of starshaped compact hypersurfaces $\Sigma$ whose principal curvatures $\kappa[\Sigma]$ are prescribed as 
\begin{equation}
\frac{\sigma_k(\eta[\Sigma])}{\sigma_{l}(\eta[\Sigma])}=\Psi(X,\nu) \quad \text{for all $X \in \Sigma$}, \label{equation 1}
\end{equation} where $X$ is the position vector, $\nu$ is the outer unit normal, and $\eta$ is the $(0,2)$-tensor defined by
\[\eta[\Sigma]=Hg_{ij}-h_{ij}, \quad \text{$H$ is the mean curvature of $\Sigma$}\] i.e. the first Newton transformation \cite[page 466]{Reilly} of the second fundamental form $h_{ij}$ with respect to the induced metric $g_{ij}$. Our main result answers the existence problem as follows.
\begin{theorem} \label{existence}
Assume $0 \leq l<k <n$.
Let $(M,g')$ be a compact Riemannian manifold and let $I$ be an open interval in $\bR$. Consider a warped product manifold $\overline{M}=I \times M$ endowed with the metric $\overline{g}=dr^2+\phi(r)^2g'$ for some positive $C^2$ function $\phi: I \to \bR$ with $\phi'>0$. Suppose $\Psi \in C^2(\Gamma)$, where $\Gamma$ is an open neighborhood of the unit normal bundle of $\Sigma$ in $\overline{M} \times \bS^n$, satisfies the following conditions:
\begin{align*}
\Psi(V,\nu)&>\frac{\binom{n}{k}}{\binom{n}{l}}[(n-1)h(r)]^{k-l}, \quad r \leq r_1, \\
\Psi(V,\nu)&<\frac{\binom{n}{k}}{\binom{n}{l}}[(n-1)h(r)]^{k-l}, \quad r \geq r_2, \\
\text{and} \quad \frac{\partial}{\partial r}[&\phi(r)^{k-l}\Psi(V,\nu)] \leq 0, \quad r_1<r<r_2,
\end{align*} where $V=\phi(r)\frac{\partial}{\partial r}$ and $h(r)=\phi'(r)/\phi(r)$. Then there exists a unique $C^{4,\alpha}$, $(\eta,k)$-convex, starshaped compact hypersurface $\Sigma$ in $\{(r,z) \in \overline{M}: r_1 \leq r \leq r_2\}$, satisfying equation \eqref{equation 1} for any $\alpha \in (0,1)$.
\end{theorem}
\begin{remark}
As we will soon review below, our theorem \ref{existence} extends recent work of Chen-Tu-Xiang \cite{Chen-Tu-Xiang-preprint} and Mei-Zhu \cite{Mei-Zhu}. A partial answer to the case when $k=n$, $l=n-1$ is also given below, while a full solution to the case remains still open.
\end{remark}
\begin{notation}
We take the convention that $\sigma_0=1$ and so when $l=0$, the curvature quotient \eqref{equation 1} would reduce to a $k$-curvature equation i.e.
\begin{equation}
\sigma_k(\eta[\Sigma])=\Psi(X,\nu). \label{equation 3}
\end{equation}
\end{notation}

\begin{remark}
A proof of purely interior curvature estimates for $k$-convex graphs over $B_r \subseteq \bR^n$ satisfying the following curvature quotient equation
\begin{equation}
\frac{\sigma_k}{\sigma_{k-1}}(\kappa[\Sigma])=\Psi(X), \quad 1 \leq k \leq n \label{curvature quotient}
\end{equation} is included in the Appendix; some might also be interested in this result (note that the input here is $\kappa[\Sigma]$, not $\eta[\Sigma]$).
\end{remark}

Motivations for studying this particular class of fully nonlinear elliptic equations i.e. \eqref{equation 1}, may come from several aspects; one major inspiration is from its close connections to problems in complex geometry. When $k=n$ and $l=0$, equation \eqref{equation 1} becomes a Monge-\Ampere\ type curvature equation
\[\det(\eta[\Sigma])=\Psi(X,\nu)\] which was intensively studied by Sha \cite{Sha-1, Sha-2}, Wu \cite{Wu} and Harvey-Lawson \cite{HL}, in order to understand the topology of manifolds with non-negative sectional curvature; see also \cite{HL-JDG, HL-Adv, HL-SDG}. In complex geometry, this equation is the so-called $(n-1)$-Monge-\Ampere\ equation, which is related to the Gauduchon conjecture \cite[Section IV.5]{Gauduchon}; the conjecture has recently been solved by Sz\'{e}kelyhidi, Tosatti and Weinkove in \cite{STW}; see also the work of Guan and Nie \cite{Guan-Nie}. For more references on this topic, the reader is referred to \cite{Fu-Wang-Wu-1, Fu-Wang-Wu-2, Popovici, TW-1, TW-2, Sze}. Another motivation is due to the study of Chern-Ricci forms; see \cite{Guan-Qiu-Yuan}.

It was then natural to study generalizations of these Monge-\Ampere\ type equations i.e. equation \eqref{equation 1}. Indeed, the problem has attracted much attention from several authors. Chu and Jiao \cite{Chu-Jiao} proved theorem \ref{existence} for the $k$-curvature equation \eqref{equation 3} in the Euclidean space $\bR^{n+1}$, and Chen-Tu-Xiang \cite{Chen-Tu-Xiang-JDE-2020} proved the same result for $(k,l)$-curvature quotients \eqref{equation 1} (still in $\bR^{n+1}$), provided that $k-l \geq 2$. Later, Zhou \cite{Zhou} extended the result of Chu-Jiao to space forms, and in a preprint \cite{Chen-Tu-Xiang-preprint}, Chen-Tu-Xiang further generalized the result for the quotient equation (which includes the $k$-curvature equation \eqref{equation 2} when $l=0$) to warped product manifolds under the same assumption of $k-l \geq 2$. 

The issue occurred when $k-l=1$, in which case a crucial control in their arguments would fall through; see e.g. \cite[(3.21)]{Chen-Tu-Xiang-JDE-2020}. Let us divide the remaining case i.e. $k-l=1$ into two sub-cases: 
\begin{align}
k&<n,\ l=k-1; \tag{A} \label{A}\\
k&=n,\ l=n-1. \tag{B} \label{B}
\end{align}
In a recent work \cite{Chen-Dong-Han}, Chen-Dong-Han derived a very convenient lemma which exploits the special structure of the equation operator, and Mei-Zhu \cite{Mei-Zhu} have applied this lemma to solve the case \ref{A} in $\bR^{n+1}$. The novelty of our theorem \ref{existence} is that we settled the case \ref{A} in warped product manifolds and thus extending both the work of Chen-Tu-Xiang \cite{Chen-Tu-Xiang-preprint} and Mei-Zhu \cite{Mei-Zhu}; one hurdle for generalizing from the Euclidean space to a warped product manifold is that the second fundamental form of the hypersurface $\Sigma$ in the warped product manifold $\overline{M}$ no longer satisfies the Codazzi property.

While the case \ref{B} remains still open, it is clear that the case could be resolved when the right-hand side function does not contain the gradient term i.e when $\Psi=\Psi(X)$ is independent of $\nu$, due to the presence of an extra positive curvature term in the interchanging formula for the fourth order term; see lemma \ref{commutator}.

\begin{theorem} \label{existence 2}
In the same setting as in theorem \ref{existence}, and under the additional assumption that $\Psi=\Psi(X)$ does not depend on $\nu$, the conclusion of theorem \ref{existence} holds true for $k=n,\ l=n-1$.
\end{theorem}
\begin{remark}
The major difficulty is how to solve the case \ref{B} when $\Psi=\Psi(X,\nu)$ depends on $\nu$ in general.
\end{remark}

Following \cite{Chen-Tu-Xiang-preprint, DBD, Jin-Li, Li-Sheng}, theorem \ref{existence} and theorem \ref{existence 2} will be proved using the standard degree theory \cite{Li}, for which we shall derive a prior estimates for admissible hypersurfaces. Since lower order estimates have already been established in \cite{Chen-Tu-Xiang-preprint}, here it is sufficient for us to obtain only the global curvature estimates. 
\begin{theorem}[global curvature estimates] \label{curvature estimate}
Let $0 \leq l<k < n$.
Suppose that $\Sigma$ is a closed star-shaped $(\eta,k)$-convex hypersurface in a warped product manifold $\overline{M}$, satisfying the curvature equation
\[\frac{\sigma_{k}(\eta[\Sigma])}{\sigma_{l}(\eta[\Sigma])}=\Psi(X,\nu) \quad \text{for all $X \in \Sigma$}\] for some positive $\Psi \in C^2(\Gamma)$, where $\Gamma$ is an open neighborhood of the unit normal bundle of $\Sigma$ in $\overline{M} \times \bS^n$. Then there exists some $C>0$ depending only on $n, k, l, \norm{\Sigma}_{C^2}, \inf \Psi, \norm{\Psi}_{C^2}$ and the curvature $\overline{R}$ of $\overline{M}$ such that 
\[\max_{X \in \Sigma}|\kappa_{\max}(X)| \leq C.\]

When $k=n$ and $l=n-1$, the same estimate holds if $\Psi=\Psi(X)$ is independent of $\nu$.
\end{theorem}

\begin{remark}
In contrast, when $1 \leq l<k \leq n$, Guan, Ren and Wang \cite{Guan-Ren-Wang} have demonstrated that the usual curvature quotient equations of principal curvatures
\[\frac{\sigma_k(\kappa[\Sigma])}{\sigma_l(\kappa[\Sigma])}=\Psi(X,\nu)\] do not admit such a curvature estimate even for strictly convex solutions; see their theorem 1.2.
\end{remark}

\subsection{Secondary results}
\indent

Moreover, as demonstrated in \cite{Guan-Ren-Wang, Chu-Jiao}, the arguments leading to curvature estimates for a prescribed curvature equation $F(\kappa[\Sigma])=\Psi$ can readily be adapted to produce $C^2$ estimates for the Hessian type equation $F(D^2u)=f$ in the same form; for our curvature equation \eqref{equation 1}, this is still the case. Consider the following Hessian type equation in the same form:
\begin{equation}
\frac{\sigma_k(\eta[u])}{\sigma_{l}(\eta[u])}=f(x,u, Du), \label{equation 2}
\end{equation} where
\[\eta[u]=(\Delta u) I - D^2u, \quad \text{$I$ is the identity matrix}.\] We have
\begin{theorem}[global $C^2$ estimates]
Let $0 \leq l<k<n$.
Suppose that $u \in C^4(\Omega) \cap C^2(\overline{\Omega})$ is a $(\eta,k)$-convex solution of 
\begin{alignat*}{2}
\frac{\sigma_k(\eta[u])}{\sigma_{l}(\eta[u])}&=f(x,u, Du) &\quad &\text{in $\Omega$,} \\
u&=\varphi &\quad &\text{on $\partial \Omega$,}
\end{alignat*} 
then there exists some $C>0$ depending only on $n, k, l, \norm{u}_{C^1}, \inf f, \norm{f}_{C^2}$ and $\Omega$ such that 
\[\sup_{\Omega} |D^2u| \leq C(1+\sup_{\partial \Omega} |D^2u|).\]

If additionally, the Dirichlet problem has a $C^3(\overline{\Omega})$ subsolution and $f_{u} \geq 0$, then for all $0 < \alpha < 1$ it admits a unique $C^{3,\alpha}$ solution $u$.
\end{theorem}
\begin{remark}
The proof for this theorem will be omitted, as it is almost identical to that of theorem \ref{curvature estimate}. 
\end{remark}

In fact, we can do better and derive a Pogorelov type interior $C^2$ estimate for equation \eqref{equation 2}.
\begin{theorem}[Pogorelov type interior $C^2$ estimates]\label{Pogorelov estimate}
Let $1 \leq k < n$ and $0 \leq l \leq k-1$.
Suppose that $u \in C^4(\Omega) \cap C^2(\overline{\Omega})$ is a $(\eta,k)$-convex solution to the following Dirichlet problem
\begin{alignat*}{2}
\frac{\sigma_k(\eta[u])}{\sigma_{l}(\eta[u])}&=f(x,u, Du) &\quad &\text{in $\Omega$,} \\
u&=0 &\quad &\text{on $\partial \Omega$.}
\end{alignat*}
Then for every $\beta>0$, there exists some $C>0$ depending only on $n,k,l,\norm{u}_{C^1}, \inf f$ and $\norm{f}_{C^2}$ such that
\[\sup_{\Omega} [\ (-u)^{\beta} |D^2u|\ ] \leq C.\]
\end{theorem}
\begin{remark}
Theorem \ref{Pogorelov estimate} was proved by Chu-Jiao \cite{Chu-Jiao} when $l=0$, and by Chen-Tu-Xiang \cite{Chen-Tu-Xiang-JDE-2021} when $k-l \geq 2$. Here our new contributions are not only that we are able to obtain the estimate in the case \ref{A} i.e. when $k-l=1$ and $k<n$, and also the exponent $\beta>0$ in our estimate can be arbitrary which holds for all $0 \leq l<k<n$. Previously in \cite{Chu-Jiao} and \cite{Chen-Tu-Xiang-JDE-2021}, the exponent $\beta$ had to be some large number. 
\end{remark}

\begin{remark}
Just like the global $C^2$ estimates, the Pogorelov type interior estimates remain still open for $k=n$, $l=n-1$.
\end{remark}

\begin{remark}
Furthermore, as pointed out by Guan-Ren-Wang \cite[Theorem 3.6]{Guan-Ren-Wang}, with the assumption of $C^1$ boundedness, we could settle an Ivochkina \cite{Ivochkina-1, Ivochkina-2, Lin-Trudinger, Ivochkina-Lin-Trudinger} type regularity problem. In fact, for this particular class of equations, we could even tackle the degenerate case; see the recent work of Jiao-Sun \cite{Jiao-Sun} and Chen-Tu-Xiang \cite{Chen-Tu-Xiang-degenerate}.
\end{remark}

Next, we consider the purely interior $C^2$ estimate for equation \eqref{equation 2}; such an estimate would usually not follow from the proof of the global $C^2$ estimates, but due to the special structure of the equation, Mei \cite{Mei} was able to obtain the estimate when $0 \leq l \leq k-2$ and $k<n$ by following the arguments of Chu-Jiao \cite{Chu-Jiao} and Chen-Tu-Xiang \cite{Chen-Tu-Xiang-JDE-2021}; he also constructed non-classical solutions when $k=n$ and $0 \leq l \leq n-3$. The remaining cases for the purely interior estimate are (1) $k-l=1$ and $k<n$; (2) $k=n$ and $l=n-1,n-2$. Here we tackle the former case.

\begin{theorem}[purely interior $C^2$ estimates]\label{interior C^2 estimates}
Let $1 \leq k < n$ and $0 \leq l \leq k-1$. Suppose that $u \in C^{\infty}(B_1)$ is an $(\eta,k)$-convex solution i.e. $\lambda(\eta[u]) \in \Gamma_k$ of the Hessian quotient equation
\[\frac{\sigma_{k}(\eta[u])}{\sigma_{l}(\eta[u])}=f(x,u,Du) \quad x \in B_1.\] Then there exists some $C>0$ depending only on $n,k,l, \inf f, \norm{f}_{C^2(B_1 \times \bR \times \bR^n)}$ and $\norm{u}_{C^1(B_1)}$ such that 
\[\sup_{B_{\frac{1}{2}}} |D^2u| \leq C.\]
\end{theorem}

\begin{remark}
Previously, the estimate for $k-l=1$ and $k<n$ was proved by Chen-Dong-Han \cite{Chen-Dong-Han} for $f=f(x)$ and by Mei \cite{Mei} for $f=f(x,u)$; here we solve the case when $f=f(x,u,Du)$, but our proof follows closely theirs and it works for all $0 \leq l < k < n$. The estimate for $k=n$ and $l=n-1,n-2$ remains still open.   
\end{remark}

Similarly, we can derive an interior curvature estimate.

\begin{theorem}
Let $0 \leq l < k < n$. Suppose that $\Sigma$ is an $(\eta,k)$-convex graph over $B_r$, satisfying
\[\frac{\sigma_k(\eta[\Sigma])}{\sigma_l(\eta[\Sigma])}=f(X,\nu), \quad X\in B_r \times \bR.\] Then there exists some $C>0$ depending on $n, k, l, r, \inf f, \norm{\Sigma}_{C^1(B_r)}$ and $\norm{f}_{C^2(B_r)}$ such that
\[\sup_{x \in B_{r/2}} |\kappa_{i}(x)| \leq C.\]
\end{theorem}

The proof for this interior curvature estimate would be exactly the same as in \cite{Zhou-CVPDE} and hence should be omitted here. However, just as the global curvature estimate, since there would be an extra positive curvature term from the the interchanging formula e.g. \eqref{commutator formula}, we can obtain the estimate when $k=n, l=n-1$ for a gradient-independent right-hand side $f=f(X)$.

\begin{theorem} \label{interior curvature estimates 2}
Let $1 \leq k \leq n$ (the result particularly holds for $k=n$). Suppose that $\Sigma$ is an $(\eta,k)$-convex graph over $B_r$, satisfying
\[\frac{\sigma_k}{\sigma_{k-1}}(\eta[\Sigma])=f(X), \quad X\in B_r \times \bR.\] Then there exists some $C>0$ depending on $n, k, l, r, \inf f, \norm{\Sigma}_{C^1(B_r)}$ and $\norm{f}_{C^2(B_r)}$ such that
\[\sup_{x \in B_{r/2}} |\kappa_{i}(x)| \leq C.\] 
\end{theorem}
The proof for this theorem would the same as the proof of theorem \ref{interior curvature estimates} in the Appendix and hence will be omitted here.

\subsection{Some related research}
\begin{remark}
With the above theorems, the $C^2$ regularity problem for this class of Hessian quotient equations \eqref{equation 2} has now been completely settled, except for the $k=n$ case(s). On the contrary, we would like to mention that a full solution to the $C^2$ regularity problem of the usual Hessian (quotient) equations
\[\frac{\sigma_k(D^2u)}{\sigma_l(D^2u)}=f(x,u,Du), \quad 0 \leq l < k \leq n,\] has been a longstanding problem. Even for the global $C^2$ estimate, only some special cases have been resolved; see the representative work of Guan-Ren-Wang \cite{Guan-Ren-Wang}, Guan-Li-Li \cite{Guan-Li-Li}, Lu \cite{Lu-CVPDE}, Ren-Wang \cite{Ren-Wang-1, Ren-Wang-2}, and Yang \cite{Yang}. For the purely interior $C^2$ estimate, the interested reader is referred to the recent substantial progress made by Guan-Qiu \cite{Guan-Qiu}, Lu \cite{Lu-1, Lu-2}, Qiu \cite{Qiu-1, Qiu-2}, Shankar-Yuan \cite{SY-preprint} and the references therein. For the Pogorelov estimate, see the work of Li-Ren-Wang \cite{Li-Ren-Wang}, Tu \cite{Tu} and Zhang \cite{Zhang}.
\end{remark}

Finally, we present one more application of the $C^2$ estimates, which is the solvability of the Dirichlet problem for an extended class of equations on Riemannian manifolds.

\begin{theorem}\label{Dirichlet problem}
Let $0 \leq l<k<n$ and let $\overline{M}:=M \cup \partial M$ be a compact Riemannian manifold of dimension $n \geq 3$, with smooth boundary $\partial M \neq \emptyset$ and $\varphi \in C^2(\partial M)$. Suppose that $\psi$ is a positive $C^\infty$ function with respect to $(x,z,p) \in \overline{M} \times \bR \times T_{x}M$, where $T_{x}M$ denotes the tangent space of $M$ at $x$, and $\psi_z>0$. Assume that the following growth condition holds: there exists a positive constant $K_0$ such that for any $x \in M', z \in \bR, p \in T_xM$ with $|p|\geq K_0$, the function $\psi$ satisfies
\[\frac{p \cdot \nabla_{x}\psi(x,z,p)}{|p|^2}-p\cdot \nabla_{p}\psi(x,z,p) \geq - \overline{\psi}(x,z)|p|^{\gamma}\] for some continuous function $\overline{\psi}>0$ of $(x,z) \in \overline{M}\times \bR$ and some constant $\gamma \in (0,2)$. If there exists some admissible subsolution $\underline{u} \in C^2(\overline{M})$, then the following Dirichlet problem 
\begin{alignat*}{2}
\frac{\sigma_k(\lambda[U])}{\sigma_{l}(\lambda[U])}&=\psi(x,u, Du) &\quad &\text{in $M$,} \\
u&=\varphi &\quad &\text{on $\partial M$.}
\end{alignat*}
admits a smooth unique admissible solution $u \in C^{\infty}(\overline{M})$ with $\lambda[U] \in \Gamma_k$, where
\[U=\theta(\Delta u)I-\mu(\nabla^2u)+\chi, \quad \text{$\chi$ is a $(0,2)$-tensor},\quad \text{and $\theta \geq \mu, \theta>0$}.\]
\end{theorem}

\begin{remark}
This theorem has been proved by Chen-Tu-Xiang \cite{Chen-Tu-Xiang-IMRN-2023} when $k-l \geq 2$, $\chi=0$ and $\mu=1$; when $\chi \neq 0$ but $\theta=\mu=1$, it has been proved by Jiao-Liu \cite{Jiao-Liu} and Liang-Zhu \cite{Liang-Zhu}. Here, our theorem \ref{Dirichlet problem} extends all their results in the sense that it holds for the wholly extended class $U=\theta(\Delta u)I-\mu(\nabla^2u)+\chi$ and all $0 \leq l \leq k-1$.

The proof will be omitted, as all a priori estimates have been well-established in \cite{Liang-Zhu} and there is no need to repeat the arguments. Instead, since the arguments build on the validity of a key lemma due to Chen-Dong-Han \cite{Chen-Dong-Han}, it suffices for us to only show the Chen-Dong-Han lemma holds for this extended class and the theorem follows as in \cite{Liang-Zhu}. The motivation for considering this generalized class is due to its complex analogue studied in \cite{Guan-Qiu-Yuan}.
\end{remark}

\begin{remark}

As pointed out by Chen-Tu-Xiang \cite{Chen-Tu-Xiang-IMRN-2023}, the growth condition can be removed when $\psi=\psi(x,u)$ is independent of $Du$, which was only used in deriving the global gradient estimate. The reader is also referred to \cite{Guan-Duke-2014, Guan-Adv-2023, Guan-CPDE-1994, Guan-CVPDE-1999, Guan-Jiao-CVPDE-2015, Guan-Jiao-DCDS-2016} for more studies of Dirichlet problems for Hessian type equations on Riemannian manifolds.
\end{remark}

We shall end the Introduction with the following brief discussions on some open problems for this class of equations.

\subsection{Open problems}
\indent

Note that the Chen-Dong-Han lemma holds for all $0 \leq l<k < n$, so our proofs here can be seen as a simple alternative of all previous ones derived for the $k-l \geq 2$ case. On the other hand, we would like to mention that the key lemma does not hold for $k=n$; see remark \ref{reason for k<n} below. Hence, all these second order estimates remain unknown for the $(n,n-1)$ quotient; this is the critical case when both the arguments of Chu-Jiao \cite{Chu-Jiao} and Chen-Tu-Xiang \cite{Chen-Tu-Xiang-IMRN-2023, Chen-Tu-Xiang-JDE-2020, Chen-Tu-Xiang-JDE-2021, Chen-Tu-Xiang-preprint}, and the key lemma of Chen-Dong-Han \cite{Chen-Dong-Han} are not applicable, thus new techniques must be employed. For the purely interior $C^2$ estimates, the $(n,n-2)$ quotient is also a critical case in this sense, because the bad term is 
\[-C\frac{\sum F^{ii}}{\rho^2}.\] Therefore, obtaining second order estimates in the critical case(s) may be an appealing problem for further study.

Another possibly appealing problem might be studying the rigidity properties of entire admissible solutions for this class of equations. Previously, Chen-Dong-Han \cite[Theorem 1.3]{Chen-Dong-Han} has done this under a two-sided growth condition. It would be desirable if one could establish the result under the preferred lower quadratic growth condition.

\subsection{Organization of this note}
\indent

The rest of this note is organized as follows. In section \ref{Preliminaries}, we introduce basic notations and standard properties of our equation operator. In section \ref{the key lemma}, we prove a generalized lemma of Chen-Dong-Han \cite{Chen-Dong-Han} for the extended class
\[\theta(\Delta u)I-\mu(D^2u),\] so that we can use it for the Dirichlet problem in theorem \ref{Dirichlet problem}. Also, all the second order estimates derived here remain valid for this extended class. This lemma is so handy that we believe many other problems which also concern this class of Hessian type equations, can be solved in similar ways by applying this lemma. In section \ref{global curvature estimate proof}, we prove the global curvature estimate; in section \ref{Pogorelov proof}, we derive the optimal Pogorelov estimates; and in section \ref{interior proof}, we obtain the purely interior $C^2$ estimate. The appendix contains a proof of interior curvature estimates for admissible graphs satisfying \eqref{curvature quotient}, from which theorem \ref{interior curvature estimates 2} follows as well.

\subsection*{Acknowledgement}

We would like to thank Professor Man-Chun Lee for bringing up the paper \cite{Chu-Jiao} of Chu-Jiao to our attention, from which we learned of the problem(s) being concerned here. We would also like to thank Professor Qiang Tu for pointing out a few mistakes in our initial draft and some communications.

\section{Preliminaries} \label{Preliminaries}
The $k$-th symmetric polynomial $\sigma_{k}(x): \bR^n \to \bR$ is a smooth symmetric function of $n$ variables, defined by
\[\sigma_k(x_1,\ldots,x_n)=\sum_{1 \leq i_1<i_2<\cdots<i_k\leq n} x_{i_1}\cdots x_{i_k}.\] The $k$-th \Garding\ cone is defined by
\[\Gamma_k=\{x \in \bR^n: \sigma_j(x)>0 \quad \forall\ 1\leq j \leq k\}.\]

\begin{notation}
For a symmetric matrix $A=(a_{ij})$, the notation $a_{ij} \in \Gamma_k$ means that its eigenvalues $\lambda(A)=(\lambda_1,\ldots,\lambda_n)$ belong to $\Gamma_k$.
\end{notation}
\begin{definition} \label{admissible solutions}
We say $\Sigma$ is an $(\eta,k)$-convex solution of \eqref{equation 1} if $\eta[\Sigma] \in \Gamma_k$. Similarly, we say $u$ is an $(\eta,k)$-convex solution of \eqref{equation 2} if $\eta[u] \in \Gamma_k$.
\end{definition}

\begin{notation}
Observe that 
\[\frac{\partial}{\partial x_i} \sigma_k(x)=\sigma_{k-1}(x)\bigg|_{x_i=0}=\sigma_{k-1}(x_1,\ldots,x_{i-1},0,x_{i+1},\ldots,x_n).\] We will denote the partial derivatives by $\sigma_{k-1}(x|i)$ or $\sigma_{k}^{ii}(x)$ interchangeably.
\end{notation}

The following three lemmas contain well-known properties of the elementary symmetric polynomials.

\begin{lemma} \label{properties of sigma-k}
For $\lambda \in \bR^n$ and $1 \leq k \leq n$, we have
\begin{align*}
\sigma_k(\lambda)&=\lambda_i \sigma_{k-1}(\lambda|i)+\sigma_k(\lambda|i), \\
\sum_{i=1}^{n} \lambda_i \sigma_{k-1}(\lambda|i)&=k\sigma_k(\lambda),\\
\sum_{i=1}^{n} \sigma_{k-1}(\lambda|i)&=(n-k+1)\sigma_{k-1}(\lambda),\\
\end{align*}
\end{lemma}

\begin{lemma} \label{Newton-Maclaurin}
Let $\lambda \in \Gamma_k$. For $n \geq k > l \geq 0$, $n \geq r>s \geq 0$ with $k \geq r$ and $l \geq s$, we have that 
\[\left[\frac{\binom{n}{k}^{-1}\sigma_k(\lambda)}{\binom{n}{l}^{-1}\sigma_l(\lambda)}\right]^{\frac{1}{k-l}} \leq \left[\frac{\binom{n}{r}^{-1}\sigma_r(\lambda)}{\binom{n}{s}^{-1}\sigma_s(\lambda)}\right]^{\frac{1}{r-s}}.\]

In particular, we have that
\[k(n-l+1)\sigma_k(\lambda)\sigma_{l-1}(\lambda) \leq l(n-k+1)\sigma_{k-1}(\lambda)\sigma_l(\lambda).\]
\end{lemma}

\begin{lemma}\label{reverse}
Suppose $\lambda=(\lambda_1,\ldots,\lambda_n) \in \Gamma_k$ is ordered as $\lambda_1 \leq \lambda_2 \leq \cdots \leq \lambda_n$. Then $\lambda_{n-k+1}>0$ and 
\[\sigma_{k-1}(\lambda|n-k+1) \geq c(n,k)\sigma_{k-1}(\lambda).\]
\end{lemma}

Let $W=(w_{ij})$ be a Codazzi tensor \footnote{We say a symmetric $2$-tensor $W$ on a Riemannian manifold $(M,g)$ is Codazzi if $W$ is closed when viewed as a $TM$-valued $1$-form i.e. $\nabla_{X} W(Y,Z) = \nabla_{Y}W(X,Z)$ for all tangent vectors $X,Y,Z$ where $\nabla$ is the Levi-Civita connection. Typical examples include the Hessian matrix $D^2u$ of a function $u \in C^2$ and the second fundamental form $(h_{ij})$ of a hypersurface.} whose eigenvalues are denoted by $\kappa=(\kappa_1,\ldots,\kappa_n)$ and let $\lambda=(\lambda_1,\ldots, \lambda_n)$ be the eigenvalues of the matrix
\[\eta=\theta \sigma_1(W)I-\mu W, \quad \text{where $\theta \geq \mu$ and $\theta>0$}.\] For the Hessian quotient in concern i.e.
\[\frac{\sigma_k(\eta[W])}{\sigma_l(\eta[W])},\] when we view it as a function of $\lambda=(\lambda_1,\ldots,\lambda_n)$, we denote it by $G(\eta)$ or $G(\lambda)$ and 
\[G^{ij}=\frac{\partial}{\partial \eta_{ij}} \left[\frac{\sigma_k(\lambda)}{\sigma_l(\lambda)}\right], \quad G^{ij,kl}=\frac{\partial^2}{\partial \eta_{ij}\partial \eta_{kl}} \left[\frac{\sigma_k(\lambda)}{\sigma_l(\lambda)}\right].\]

Similarly, when we view the quotient as a function of $W=(w_{ij})$, we denote it by $F(W)$ or $F(\kappa)$ and 
\[F^{ij}=\frac{\partial}{\partial w_{ij}} \left[\frac{\sigma_k(\eta)}{\sigma_l(\eta)}\right], \quad F^{ij,kl}=\frac{\partial^2}{\partial w_{ij}\partial w_{kl}} \left[\frac{\sigma_k(\eta)}{\sigma_l(\eta)}\right].\]

Throughout, we assume $W=(w_{ij})$ is diagonal and 
\[\kappa_1 \geq \kappa_2 \geq \cdots \geq \kappa_n, \quad F^{11} \leq F^{22} \leq \cdots \leq F^{nn}. \] Thus, we have $\lambda_1 \leq \lambda_2 \leq \cdots \leq \lambda_n$ and
\[G^{11} \geq G^{22} \geq \cdots \geq G^{nn}.\]
The eigenvalues and the derivatives are related by
\[\lambda_i=\theta \sum_{j = 1}^{n} \kappa_j-\mu\kappa_i, \quad F^{ii}=\theta \sum_{j = 1}^{n} G^{jj}- \mu G^{ii}.\] Moreover, it can be readily verified that
\begin{align*}
\sum F^{ii}w_{ii}&=\sum G^{ii} \eta_{ii}, \\
\sum F^{ii}w_{iih}&=\sum G^{ii} \eta_{iih}, \\
\sum F^{ii}w_{iihh}&=\sum G^{ii} \eta_{iihh},
\end{align*} see e.g. \cite[(2.7)-(2.8)]{Mei}.

Since we will be working with $(\eta,k)$-convex solutions i.e. $\lambda \in \Gamma_k$, the pair $(G(\lambda),\Gamma_k)$ is the usual Hessian quotient operator with eigenvalues ordered reversely. Hence it satisfies the usual well-known properties and in particular the following:
\begin{lemma}[Concavity] \label{concavity}
If $\lambda \in \Gamma_k$ and $0 \leq l<k \leq n$, then
\[\sum_{i,j}\left[\frac{\partial^2}{\partial \lambda_i\partial \lambda_j}\frac{\sigma_k(\lambda)}{\sigma_l(\lambda)}\right]\xi_i\xi_j \leq \left(1-\frac{1}{k-l}\right)\frac{\left(\sum_{i}\left[\frac{\partial}{\partial \lambda_i}\frac{\sigma_k(\lambda)}{\sigma_l(\lambda)}\right]\xi_i\right)^2}{\frac{\sigma_k(\lambda)}{\sigma_l(\lambda)}}\] for any $\xi=(\xi_1,\ldots,\xi_n) \in \bR^n$.
\end{lemma}
\begin{proof}
See proposition 2.6 in \cite{Chen-Dong-Han}.
\end{proof}
\section{The Key Lemma} \label{the key lemma}
In this section, we provide a modest extension of the key lemma \cite[Lemma 2.7(5)]{Chen-Dong-Han} due to Chen-Dong-Han.

\begin{lemma} \label{key lemma}
Let $0 \leq l < k < n$ and $\theta \geq \mu$ with $\theta>0$. We have that
\[F^{11} \geq c(n,k,l,\theta,\mu) \sum_{i=1}^{n} F^{ii} \geq c(n,k,l,\theta,\mu, \inf F).\]
\end{lemma}
\begin{remark}
This lemma was proved by Chen-Dong-Han in \cite{Chen-Dong-Han}, for $\theta=\mu=1$. It is plausible to think that the property should still hold for general $\theta,\mu$ satisfying $\theta \geq \mu$ and $\theta>0$, due to the linear structure. However, since it is one of the most essential ingredients in our derivations of a priori estimates, we shall provide a solid verification here.
\end{remark}

\begin{proof}
Our proof is adapted from \cite{Chen-Dong-Han} and is based on the following three basic facts: $\lambda \in \Gamma_k$, $\lambda_1 \leq \lambda_2 \leq \cdots \leq \lambda_n$ and
\[\frac{\partial \lambda_j}{\partial \kappa_i}=\theta-\mu \delta_{ij}.\] 

By a direct computation and lemma \ref{properties of sigma-k}, we have that
\begin{align*}
\sum_i \frac{\partial}{\partial \kappa_i} \left[\frac{\sigma_k(\lambda)}{\sigma_l(\lambda)}\right]&=\sum_i\sum_j\frac{\partial}{\partial \lambda_j} \left[\frac{\sigma_k(\lambda)}{\sigma_l(\lambda)}\right]\frac{\partial \lambda_j}{\partial \kappa_i} \\
&=(\theta n - \mu)\sum_{j=1}^{n} \frac{\partial}{\partial \lambda_j} \left[\frac{\sigma_k(\lambda)}{\sigma_l(\lambda)}\right] \\
&=(\theta n - \mu) \sum_{j=1}^{n} \frac{\sigma_{k-1}(\lambda|j)\sigma_l(\lambda)-\sigma_k(\lambda)\sigma_{l-1}(\lambda|j)}{\sigma_l(\lambda)^2} \\
&=(\theta n -\mu) \frac{(n-k+1)\sigma_{k-1}(\lambda)\sigma_l(\lambda)-(n-l+1)\sigma_k(\lambda)\sigma_{l-1}(\lambda)}{\sigma_l(\lambda)^2} \\
&\geq (\theta n -\mu) \frac{(n-k+1)\sigma_{k-1}(\lambda)\sigma_l(\lambda)-\frac{l}{k}(n-k+1)\sigma_{k-1}(\lambda)\sigma_{l}(\lambda)}{\sigma_l(\lambda)^2}\\
&=(\theta n -\mu)(n-k+1)\left(1-\frac{l}{k}\right)\frac{\sigma_{k-1}(\lambda)}{\sigma_l(\lambda)} \\
&\geq c(n,k,l,\theta,\mu)\left[\frac{\sigma_k(\lambda)}{\sigma_l(\lambda)}\right]^{1-\frac{1}{k-l}}
\end{align*} where we have applied lemma \ref{Newton-Maclaurin} twice; this proves the second inequality.

The other inequality needs a little bit more efforts. By the same computation, we have that 
\begin{align*}
\sum_i \frac{\partial}{\partial \kappa_i} \left[\frac{\sigma_k(\lambda)}{\sigma_l(\lambda)}\right]&=(\theta n -\mu) \frac{(n-k+1)\sigma_{k-1}(\lambda)\sigma_l(\lambda)-(n-l+1)\sigma_k(\lambda)\sigma_{l-1}(\lambda)}{\sigma_l(\lambda)^2} \\
&\leq (\theta n -\mu)(n-k+1)\frac{\sigma_{k-1}(\lambda)}{\sigma_l(\lambda)}
\end{align*} by using $l < k$ and the fact that $\sigma_j(\lambda)>0$ for $1 \leq j \leq k$.

The task is then reduced to show that 
\[\frac{\partial}{\partial \kappa_1}\left[\frac{\sigma_k(\lambda)}{\sigma_l(\lambda)}\right]\geq c(n,k,l,\theta, \mu) \frac{\sigma_{k-1}(\lambda)}{\sigma_l(\lambda)}.\]

Note that
\begin{align*}
\frac{\partial}{\partial \kappa_i}\left[\frac{\sigma_k(\lambda)}{\sigma_l(\lambda)}\right]&=\sum_{j} \frac{\partial}{\partial \lambda_j}\left[\frac{\sigma_k(\lambda)}{\sigma_l(\lambda)}\right]\frac{\partial \lambda_j}{\partial \kappa_i} \\
&=\theta\sum_{j \neq i}\frac{\partial}{\partial \lambda_j}\left[\frac{\sigma_k(\lambda)}{\sigma_l(\lambda)}\right]+\underbrace{(\theta-\mu)\frac{\partial}{\partial \lambda_i}\left[\frac{\sigma_k(\lambda)}{\sigma_l(\lambda)}\right]}_{\geq 0}\\
&\geq \theta\sum_{j \neq i}\frac{\partial}{\partial \lambda_j}\left[\frac{\sigma_k(\lambda)}{\sigma_l(\lambda)}\right] \\
&=\theta \sum_{j \neq i}\frac{\sigma_{k-1}(\lambda|j)\sigma_l(\lambda)-\sigma_k(\lambda)\sigma_{l-1}(\lambda|j)}{\sigma_{l}(\lambda)^2} \\
&=\theta \sum_{j \neq i}\frac{\sigma_{k-1}(\lambda|j)\sigma_l(\lambda|j)-\sigma_k(\lambda|j)\sigma_{l-1}(\lambda|j)}{\sigma_{l}(\lambda)^2}\\
\end{align*} where we have used lemma \ref{properties of sigma-k} to expand $\sigma_l$ and $\sigma_k$. Now, for each $j$, we have $\sigma_{l-1}(\lambda|j)>0$ since $\lambda \in \Gamma_k$. If $(\lambda|j) \not\in \Gamma_k$, then $\sigma_k(\lambda|j) \leq 0$ and it immediately follows that
\[\frac{\partial}{\partial \kappa_i}\left[\frac{\sigma_k(\lambda)}{\sigma_l(\lambda)}\right] \geq \theta \sum_{j \neq i}\frac{\sigma_{k-1}(\lambda|j)\sigma_l(\lambda|j)}{\sigma_{l}(\lambda)^2}, \quad 1 \leq i \leq n.\] While if $(\lambda|j) \in \Gamma_k$, then by lemma \ref{Newton-Maclaurin}, we have that
\begin{align*}
\frac{\partial}{\partial \kappa_i}\left[\frac{\sigma_k(\lambda)}{\sigma_l(\lambda)}\right]&=\theta \sum_{j \neq i}\frac{\sigma_{k-1}(\lambda|j)\sigma_l(\lambda|j)-\sigma_k(\lambda|j)\sigma_{l-1}(\lambda|j)}{\sigma_{l}(\lambda)^2} \\
&\geq \theta \sum_{j \neq i}\frac{\sigma_{k-1}(\lambda|j)\sigma_l(\lambda|j)-\frac{l(n-k+1)}{k(n-l+1)}\sigma_{k-1}(\lambda|j)\sigma_{l}(\lambda|j)}{\sigma_{l}(\lambda)^2}\\
&=\theta \left[1-\frac{l(n-k+1)}{k(n-l+1)}\right]\sum_{j \neq i}\frac{\sigma_{k-1}(\lambda|i)\sigma_l(\lambda|j)}{\sigma_l(\lambda)^2}.
\end{align*}

In either case, we have that
\[\frac{\partial}{\partial \kappa_i}\left[\frac{\sigma_k(\lambda)}{\sigma_l(\lambda)}\right] \geq c(n,k,l,\theta) \sum_{j \neq i}\frac{\sigma_{k-1}(\lambda|j)\sigma_l(\lambda|j)}{\sigma_{l}(\lambda)^2}, \quad 1 \leq i \leq n.\]

Finally, since $0 \leq l<k<n$, we have that 
\begin{equation}
n-l \geq n-k+1 \geq 2. \label{k<n}
\end{equation} Hence, by lemma \ref{reverse}, we have that
\begin{align*}
&\sigma_{k-1}(\lambda|2) \geq \sigma_{k-1}(\lambda|n-k+1) \geq c(n,k)\sigma_{k-1}(\lambda)\\
\quad \text{and} \quad &\sigma_l(\lambda|2) \geq \sigma_{l}(\lambda|n-l) \geq c(n,l)\sigma_l(\lambda).
\end{align*} Therefore, 
\begin{align*}
\frac{\partial}{\partial \kappa_1}\left[\frac{\sigma_k(\lambda)}{\sigma_l(\lambda)}\right] &\geq c(n,k,l,\theta) \sum_{j \neq 1}\frac{\sigma_{k-1}(\lambda|j)\sigma_l(\lambda|j)}{\sigma_{l}(\lambda)^2} \\
&\geq c(n,k,l,\theta) \frac{\sigma_{k-1}(\lambda|2)\sigma_l(\lambda|2)}{\sigma_{l}(\lambda)^2} \\
&\geq c(n,k,l,\theta) \frac{\sigma_{k-1}(\lambda)}{\sigma_{l}(\lambda)}
\end{align*} and the lemma is proved; note that the constant $c(n,k,l,\theta)$ may change from line to line but is still denoted by the same symbol.
\end{proof}

\begin{remark} \label{reason for k<n}
The condition $k<n$ is used in \eqref{k<n}.
\end{remark}

\begin{remark}
It may be worth noting that, in general, even when $k=n$, we always have that
\[F^{ii} \geq c(n,k,l,\theta,\mu) \sum_{j} F^{jj} \quad \text{for $i \geq 2$}.\] Indeed, we use lemma \ref{properties of sigma-k} to obtain that
\[n\sigma_{k-1}(\lambda|1) \geq \sum_{i=1}^{n} \sigma_{k-1}(\lambda|i)=(n-k+1)\sigma_{k-1}(\lambda)\] and so
\begin{align*}
\frac{\partial}{\partial \kappa_i}\left[\frac{\sigma_k(\lambda)}{\sigma_l(\lambda)}\right] &\geq c(n,k,l,\theta) \sum_{j \neq i}\frac{\sigma_{k-1}(\lambda|j)\sigma_l(\lambda|j)}{\sigma_{l}(\lambda)^2}\\
&\geq c(n,k,l,\theta) \frac{\sigma_{k-1}(\lambda|1)\sigma_l(\lambda|1)}{\sigma_{l}(\lambda)^2} \\
&\geq c(n,k,l,\theta)\frac{n-k+1}{n}\frac{n-l}{n}\frac{\sigma_{k-1}(\lambda)}{\sigma_{l}(\lambda)} \\
&=c(n,k,l,\theta)\frac{\sigma_{k-1}(\lambda)}{\sigma_l(\lambda)}.
\end{align*}
\end{remark}

\section{The Global Curvature Estimate} \label{global curvature estimate proof}
In this section, we obtain the global curvature estimate for equation \eqref{equation 1} with a general right-hand side $\Psi=\Psi(X,\nu)$ when $0 \leq l<k<n$, and with a gradient-independent right-hand side $\Psi=\Psi(X)$ when $k=n$, $l=n-1$, hence proving theorem \ref{curvature estimate}. As reviewed in the Introduction, some cases are already solved in the work of Chu-Jiao \cite{Chu-Jiao} and Chen-Tu-Xiang \cite{Chen-Tu-Xiang-JDE-2020, Chen-Tu-Xiang-preprint}, here the new contributions are the following.

\begin{theorem}
Let $1 \leq k<n$ and assume the same settings as in theorem \ref{existence}.
Suppose that $\Sigma$ is a closed star-shaped $(\eta,k)$-convex hypersurface in a warped product manifold $\overline{M}$, satisfying the curvature equation
\[\frac{\sigma_{k}}{\sigma_{k-1}}(\eta[\Sigma])=\Psi(X,\nu) \quad \text{for all $X \in \Sigma$}\] for some positive $\Psi \in C^2(\Gamma)$, where $\Gamma$ is an open neighborhood of the unit normal bundle of $\Sigma$ in $\overline{M} \times \bS^n$. Then there exists some $C>0$ depending only on $n, k, \norm{\Sigma}_{C^2}, \inf \Psi, \norm{\Psi}_{C^2}$ and the curvature $\overline{R}$ of $\overline{M}$ such that 
\[\max_{X \in \Sigma}|\kappa_{\max}(X)| \leq C.\]
\end{theorem}
\begin{theorem}
Assume the same settings as in theorem \ref{existence}.
Suppose that $\Sigma$ is a closed star-shaped $(\eta,k)$-convex hypersurface in a warped product manifold $\overline{M}$, satisfying the curvature equation
\[\frac{\sigma_{n}}{\sigma_{n-1}}(\eta[\Sigma])=\Psi(X) \quad \text{for all $X \in \Sigma$}\] for some positive $\Psi \in C^2(\Gamma)$, where $\Gamma$ is an open neighborhood of the unit normal bundle of $\Sigma$ in $\overline{M} \times \bS^n$. Then there exists some $C>0$ depending only on $n, k, \norm{\Sigma}_{C^2}, \inf \Psi, \norm{\Psi}_{C^2}$ and the curvature $\overline{R}$ of $\overline{M}$ such that 
\[\max_{X \in \Sigma}|\kappa_{\max}(X)| \leq C.\]
\end{theorem}
\begin{remark}
When $k=n$, $l=n-1$, the estimate is currently still unknown for a general right-hand side $\Psi=\Psi(X,\nu)$.
\end{remark}

Before we begin the proof, we first review the geometry of hypersurfaces in a warped product manifold. The contents are extracted from \cite{Chen-Li-Wang, Chen-Tu-Xiang-preprint, Guan-Li-Wang, DBD}, but note the notational differences in here.

Let $(M,g')$ be a compact Riemannian manifold and let $I=(b_0,b), b \leq \infty$ be an open interval in $\bR$. We consider the warped product manifold $\overline{M}=I \times_{\phi} M$ endowed with the metric
\[\overline{g}=dr^2+\phi^2(r)g'\] where $\phi: I \to \bR$ is a smooth positive function.

The metric in $\overline{M}$ is denoted by $\langle \cdot,\cdot\rangle$ and the corresponding connection is denoted by $\overline{\nabla}$. The usual connection in $M$ is denoted by $\nabla'$. The curvature tensors in $M$ and $\overline{M}$ are denoted by $R$ and $\overline{R}$, respectively.

Let $\{e_1,\ldots,e_n\}$ be an orthonormal frame field in $M$. An orthonormal frame in $\overline{M}$ may be defined by $\overline{e}_i=\frac{1}{\phi}e_i, 1 \leq i \leq n$ and $\overline{e}_0=\frac{\partial}{\partial r}$.

A star-shaped compact hypersurface $\Sigma$ in $\overline{M}$ can be represented as a smooth radial graph 
\[\Sigma=\{X(z)=(r(z),z): z \in M\}\] for a smooth function $r: M \to I$, whose tangent space is spanned at each point by
\[X_i=\phi\overline{e}_i+r_i\overline{e}_0,\] where $r_i$ are the components of the differential $dr=r_i\theta^i$. The outward unit normal is given by 
\[\nu=\frac{1}{\sqrt{\phi^2+|\nabla' r|^2}}\left(\phi\overline{e}_0-\sum r^i \overline{e}_i\right),\] where $\nabla'r=\sum r^ie_i$.

\begin{lemma}\label{commutator}
Let $X_0$ be a point of $\Sigma$ and $\{E_0=\nu, E_1,\ldots,E_n\}$ be an adapted frame field such that each $E_i$ is a principal direction and connection forms $\omega_{i}^{k}=0$ at $X_0$. Let $h_{ij}$ be the second fundamental form of $\Sigma$. Then at $X_0$, we have that
\begin{align*}
h_{ijk}=  &\ h_{ikj}+\overline{R}_{0ijk}\\
h_{ii11}= &\ h_{11ii}+h_{11}h_{ii}^2-h_{11}^2h_{ii}+2(h_{ii}-h_{11})\overline{R}_{i1i1}\\
&+h_{11}\overline{R}_{i0i0}-h_{ii}\overline{R}_{1010}+\overline{R}_{i1i0;1}-\overline{R}_{1i10;i}.
\end{align*}
\end{lemma}

\begin{remark}
The frame field $E_i$ may be obtained from the adapted frame field $\nu, X_1,\ldots,X_n$ by the Gram-Schmidt procedure. Since this last frame depends only on $r$ and $\nabla' r$, we may say that the components of $\overline{R}$ and $\overline{\nabla}\bar{R}$ calculated in terms of the frame $\{E_i\}$ depend only on $r$ and $\nabla' r$.
\end{remark}

The position vector $V=\phi(r)\partial_r$ is a conformal Killing vector field and the support function is given by $\tau=\langle V,\nu\rangle$ where $\nu$ is the outward unit normal of $\Sigma$. Also, we define 
\[\Phi(r)=\int_{0}^{r} \phi(\rho)\ d\rho.\]

\begin{lemma}\label{geometric formulas}
Let $\tau$ be defined as above. Then at a local orthonormal frame, we have 
\begin{align*}
\nabla_{E_i}\Phi&=\phi\langle \overline{e}_0,E_i\rangle E_i,\\
\nabla_{E_i}\tau&=\sum_{j}(\nabla_{E_j} \Phi)h_{ij},\\
\nabla_{E_i,E_j}^{2}\tau&=\sum_{k}(h_{ijk}-\overline{R}_{0ijk})\nabla_{E_k}\Phi+\phi'h_{ij}-\tau\sum_{k}h_{ik}h_{kj},\\
\nabla_{E_i,E_j}^2\Phi&=\phi'(r)g_{ij}-\tau h_{ij}.
\end{align*}
\end{lemma}

We will work with the operators $F$ and $G$ as fixed in the Preliminaries 
\begin{equation}
F(\kappa)=G(\lambda)=\frac{\sigma_k(\eta[\Sigma])}{\sigma_l(\eta[\Sigma])}=\Psi(V,\nu). \label{curvature equation}
\end{equation} Recall that
\[G^{ij}=\frac{\partial}{\partial \eta_{ij}} \left[\frac{\sigma_k(\eta[\Sigma])}{\sigma_l(\eta[\Sigma])}\right] \quad \text{and} \quad F^{ij}=\frac{\partial}{\partial h_{ij}} \left[\frac{\sigma_k(\eta[\Sigma])}{\sigma_l(\eta[\Sigma])}\right].\]

Due to the assumption of a $C^1$ estimate, there exists some $C>0$ depending on $\inf_{\Sigma} r$ and $\norm{r}_{C^1}$ such that
\[\frac{1}{C} \leq \inf_{\Sigma} \tau \leq \tau \leq \sup_{\Sigma} \tau \leq C.\]

\begin{proof}[Proof of theorem \ref{curvature estimate}]
We first consider the case $0 \leq l < k <n$ for a general right-hand side $\Psi=\Psi(X,\nu)$.

\subsection*{When $0<l<k<n$ and $\Psi=\Psi(X,\nu)$}
Let $\kappa_{\max}$ denote the largest principal curvature of $\Sigma$. Since $\eta \in \Gamma_{k} \subseteq \Gamma_{1}$, we have that 
\[(n-1)\sigma_1(\kappa)=\sigma_1(\eta)>0.\] Hence it is sufficient to prove $\kappa_{\max}$ is bounded from above. We consider the following test function 
\[Q=\log \kappa_{\max} - \log (\tau-a)\] where $a>0$ is some number satisfying $a<\frac{1}{2}\inf \tau$.

Suppose that $Q$ attains its maximum at some point $X_0$. We choose a local orthonormal frame $\{E_1,\ldots,E_n\}$ around $X_0$ such that
\[h_{ij}=\kappa_i\delta_{ij} \quad \text{and} \quad \kappa_{\max}=\kappa_1\geq \kappa_2 \geq \cdots \geq \kappa_n.\]

In case that $\kappa_1$ has multiplicity $m>1$ i.e.
\[\kappa_1=\kappa_2=\cdots = \kappa_m>\kappa_{m+1} \geq \cdots \geq \kappa_n,\] we may apply \footnote{Alternatively, one may also apply a standard perturbation argument; see e.g. \cite{Chu,Bin}.} a smooth approximation lemma \cite[Lemma 5]{BCD} of Brendle-Choi-Daskalopoulos to have
\begin{align}
\delta_{kl}\cdot (\kappa_{1})_i&=h_{kli}, \quad 1 \leq k,l\leq m \label{curvature approximation 1}\\
(\kappa_1)_{ii}&\geq h_{11ii}+2\sum_{p>m}\frac{h_{1pi}^2}{\kappa_1-\kappa_p}. \nonumber
\end{align} Thus, at $X_0$, we have

\begin{align}
0&=\frac{(\kappa_1)_i}{\kappa_1}-\frac{\tau_i}{\tau-a}=\frac{h_{11i}}{\kappa_1}-\frac{\tau_i}{\tau-a}, \label{1st order critical}\\
0&\geq  \frac{(\kappa_1)_{ii}}{\kappa_1}-\frac{(\kappa_1)_{i}^2}{\kappa_{1}^2}-\frac{\tau_{ii}}{\tau-a}+\frac{\tau_{i}^2}{(\tau-a)^2} \nonumber\\
&\geq \frac{h_{11ii}}{\kappa_1}-\frac{\tau_{ii}}{\tau-a}. \label{2nd order critical}
\end{align}
Contracting \eqref{2nd order critical} with $F^{ii}$, we have
\begin{equation}\label{2nd order critical 1}
0 \geq \frac{F^{ii} h_{11ii}}{\kappa_1}-\frac{F^{ii} \tau_{ii}}{\tau-a}.
\end{equation}

By the commutator formula (lemma \ref{commutator}), we have 
\begin{gather} \label{interchanging}
\begin{split}
h_{ii11}=&\ h_{11ii}+h_{11}h_{ii}^2-h_{11}^2h_{ii}+2(h_{ii}-h_{11})\overline{R}_{i1i1} \\
&+h_{11}\overline{R}_{i0i0}-h_{ii}\overline{R}_{1010}+\overline{R}_{i1i0;1}-\overline{R}_{1i10;i}
\end{split}
\end{gather} from which it follows that 
\begin{align*}
F^{ii}h_{11ii}&=F^{ii}h_{ii11}-\kappa_1\sum F^{ii}\kappa_{i}^2+\kappa_{1}^2 \sum F^{ii} \kappa_i+2\sum F^{ii}(h_{11}-h_{ii})\overline{R}_{i1i1}\\
&\quad -\kappa_1\sum F^{ii} \overline{R}_{i0i0}+\sum F^{ii} h_{ii}\overline{R}_{1010}+\sum F^{ii}(\overline{R}_{1i10;i}-\overline{R}_{i1i0;1})\\
&\geq F^{ii}h_{ii11}-\kappa_1\sum F^{ii} \kappa_{i}^2 - C\kappa_1\sum F^{ii}-C
\end{align*} for some $C>0$ depending on $\inf \Psi$ and $\overline{R}$, where we have used 
\begin{align*}
\sum F^{ii}\kappa_i&=\sum_i \left(\sum_{k} G^{kk}-G^{ii}\right)\left(\frac{1}{n-1}\sum_{l} \lambda_l - \lambda_i\right) \\
&=\sum_{i} G^{ii}\lambda_i = G = \Psi \quad \text{by lemma 2.6 in \cite{Chu-Dinew}.}
\end{align*}

To estimate the $h_{ii11}$ term, we differentiate equation \eqref{curvature equation} twice, and we have 
\begin{equation}
F^{ii}h_{ii1}=G^{ii}\eta_{ii1}=\Psi_1=h_{11}d_{\nu}\Psi(E_1)+d_{V}\Psi(\nabla_{E_1}V) \label{curvature differentiate once}
\end{equation} and
\begin{align*}
G^{ii}\eta_{ii11}+G^{ij,kl}\eta_{ij1}\eta_{kl1}&=\Psi_{11} \\
&=d_{V}^2\Psi(\nabla_{E_1}V,\nabla_{E_1}V)+d_V\Psi(\nabla_{E_1,E_1}^{2}V) \\
&\ + 2d_Vd_\nu \Psi(\nabla_{E_1}V,\nabla_{E_1}\nu)+d_{\nu}^2\Psi(\nabla_{E_1}\nu,\nabla_{E_1}\nu)+d_{\nu}\Psi(\nabla_{E_1,E_1}^2\nu)\\
&\geq \sum_{i} h_{1i1}(d_{\nu}\Psi)(E_i)-C\kappa_{1}^2-C\kappa_1-C \\
&= \sum_{i} h_{11i}(d_{\nu}\Psi)(E_i)+\sum_i \overline{R}_{01i1}(d_{\nu}\Psi)(E_i)-C\kappa_{1}^2-C\kappa_1-C \\
&\geq \sum_{i} h_{11i}(d_{\nu}\Psi)(E_i)-C\kappa_{1}^2-C\kappa_1-C
\end{align*} for some $C>0$ depending on $\norm{\Psi}_{C^2}$ and the curvature $\overline{R}$.

Note that
\[-G^{ij,kl}\eta_{ij1}\eta_{kl1}=-G^{pp,qq}\eta_{pp1}\eta_{qq1}+G^{pp,qq}\eta_{pq1}^2,\] where the first term
\[-G^{pp,qq}\eta_{pp1}\eta_{qq1} \geq -\frac{C}{G}\left(\sum G^{ii}\eta_{ii1}\right)^2 \geq -C\kappa_{1}^2\] by concavity of $G$ i.e. lemma \ref{concavity} and \eqref{curvature differentiate once}, and the second term
\[G^{pp,qq}\eta_{pq1}^2\geq 2\sum_{i\neq 1}G^{11,ii}\eta_{1i1}^2 =2\sum_{i\neq 1}\frac{G^{ii}-G^{11}}{\eta_{11}-\eta_{ii}}\eta_{11i}^2\geq 0.\]
Therefore,
\begin{align*}
\frac{F^{ii}h_{ii11}}{\kappa_1}=\frac{G^{ii}\eta_{ii11}}{\kappa_1} &= \frac{-G^{ij,kl}\eta_{ij1}\eta_{kl1}}{\kappa_1}+\frac{\Psi_{11}}{\kappa_1}\\
&\geq \sum_{i} \frac{h_{11i}}{\kappa_1}(d_{\nu}\Psi)(E_i)-C\kappa_{1}-C.
\end{align*} and \eqref{2nd order critical 1} becomes
\begin{gather} \label{2nd order critical 2}
\begin{split}
0&\geq \sum_{i} \frac{h_{11i}}{\kappa_1}(d_{\nu}\Psi)(E_i)-\frac{F^{ii} \tau_{ii}}{\tau-a}-\sum F^{ii} \kappa_{i}^2 - C \sum F^{ii} -C\kappa_{1}-C.
\end{split}
\end{gather}

Recall from lemma \ref{geometric formulas} that
\begin{align*}
\tau_{ii}&=\sum_{k}(h_{iik}-\overline{R}_{0iik})\Phi_k+\phi'\kappa_i-\tau\kappa_{i}^2.\\
\end{align*} Now, substituting this into \eqref{2nd order critical 2}, we have
\begin{gather} \label{2nd order critical 3}
\begin{split}
0 \geq &\sum_{i} \frac{h_{11i}}{\kappa_1}(d_{\nu}\overline{\Psi})(E_i)-\sum F^{ii} \kappa_{i}^2 - C \sum F^{ii} -C\kappa_{1}-C\\
& -\frac{1}{\tau-a} \sum F^{ii} \left[\sum_k (h_{iik}-\overline{R}_{0iik})\Phi_k+\phi'\kappa_i-\tau\kappa_{i}^2\right] \\
\geq &\ \frac{a}{\tau-a}\sum F^{ii} \kappa_{i}^2 +\left[\sum_{k} \frac{h_{11k}}{\kappa_1}(d_{\nu}\Psi)(E_k)-\frac{F_k\Phi_k}{\tau-a}\right]-C\sum F^{ii} - C\kappa_1 - C\\
\end{split}
\end{gather}
By \eqref{1st order critical} , \eqref{curvature differentiate once} and lemma \ref{geometric formulas}, we have
\begin{align*}
&\sum_{k} \frac{h_{11k}}{\kappa_1}(d_{\nu}\Psi)(E_k)-\frac{F_k\Phi_k}{\tau-a} \\
=&\sum_{k} \frac{\tau_k}{\tau - a} (d_{\nu} \Psi)(E_k)-\left[h_{kk}(d_\nu\Psi)(E_k)+(d_V\Psi)(X_k)\right]\frac{\Phi_k}{\tau-a} \\
\geq & \sum_{k} \frac{h_{kk}\Phi_k}{\tau - a}(d_\nu \Psi)(E_k) - h_{kk}(d_\nu \Psi)(E_k)\frac{\Phi_k}{\tau-a}-C \\
= & - C
\end{align*}

Thus, \eqref{2nd order critical 3} implies that
\begin{gather} \label{2nd order critical 4}
\begin{split}
0 \geq &\ \frac{a}{\tau-a}\sum F^{ii} \kappa_{i}^2 -C\sum F^{ii} -C\kappa_1 -C.
\end{split}
\end{gather} The rest is apparent due to the key lemma. However, for the convenience of deriving subsequent second order estimates, we shall establish the following.

\begin{lemma} \label{key control}
Let $0 \leq l<k<n$.
For our operator $F$, we have that
\[\frac{a}{\tau-a}\sum F^{ii}\kappa_{i}^2 -C\sum F^{ii} - C\kappa_1 \geq\left( Ca \sum F^{ii}\right) \kappa_{1}^2\] by assuming $\kappa_1$ is sufficiently large.
\end{lemma}
\begin{proof}[Proof of Lemma \ref{key control}]
Indeed, by lemma \ref{key lemma}, we have that
\begin{align*}
&\ \frac{a}{\tau-a}\sum F^{ii}\kappa_{i}^2-C\sum F^{ii}- C\kappa_1 \\
\geq &\ \frac{a}{\tau-a} F^{11}\kappa_{1}^2-C\sum F^{ii}-C\kappa_1 \\
\geq &\ \frac{a}{\tau-a}\left(c_0\sum F^{ii}\right)\kappa_{1}^2-C\sum F^{ii}-\left(C\sum F^{ii}\right)\kappa_1\\
\geq &\ C\sum F^{ii} \left(Ca\kappa_{1}^2-C\kappa_1-C\right)\\
\geq &\ \left(Ca\sum F^{ii}\right)\kappa_{1}^2
\end{align*} by assuming $\kappa_1$ sufficiently large.
\end{proof}
With this lemma, \eqref{2nd order critical 4} implies
\[0 \geq \left(C a\sum F^{ii}\right) \kappa_{1}^2 - C\sum F^{ii}\]
and the desired estimate follows by dividing $\sum F^{ii}$.

\subsection*{When $k=n,l=n-1$ and $\Psi=\Psi(X)$}
In this case, the twice differentiation $\Psi_{11}$ would only give rise a negative constant $-C$ instead of a negative curvature term $-C\kappa_1$; and for the particular curvature quotient operator, we have $\sum F^{ii} \leq C$. Also, there would be an extra positive curvature term 
\[\kappa_{1}^2\sum_{i=1}^{n} F^{ii}\kappa_i \geq C(n,\inf F) \kappa_{1}^2\]
from the interchanging formula \eqref{interchanging}, which were neglected in the above analysis but could be utilized here. The second order equation would thus become
\[0 \geq C\kappa_1 - C\sum F^{ii} - C \geq C\kappa_1 - C\] and the estimate follows.
\end{proof}

\section{The Pogorelov Type Estimate} \label{Pogorelov proof}

In this section, we obtain the Pogorelov type interior $C^2$ estimate for equation \eqref{equation 1} when $0 \leq l<k<n$, hence proving theorem \ref{Pogorelov estimate}. As reviewed in the Introduction, the $k-l \geq 2$ case has already been solved in \cite{Chu-Jiao, Chen-Tu-Xiang-JDE-2021}; here the new contribution is the following theorem.

\begin{theorem}
Let $1 \leq k < n$.
Suppose $u \in C^4(\Omega) \cap C^2(\overline{\Omega})$ is a $(\eta,k)$-convex solution to the following Dirichlet problem
\begin{alignat*}{2}
\frac{\sigma_k}{\sigma_{k-1}}(\eta[u])&=f(x,u, Du) &\quad &\text{in $\Omega$,} \\
u&=0 &\quad &\text{on $\partial \Omega$.}
\end{alignat*} Then for every $\beta>0$, there exists some $C>0$ depending on $n, k, \norm{u}_{C^1}, \norm{f}_{C^2}$ and $\inf f$ such that
\[\sup_{\Omega}\ (-u)^{\beta} |D^2u| \leq C.\]
\end{theorem}

\begin{remark}
Our proof works for all $0\leq l<k<n$ and the exponent $\beta>0$ could be arbitrary; previously in \cite{Chu-Jiao, Chen-Tu-Xiang-JDE-2021}, the number had to be possibly large.
\end{remark}

\begin{remark}
When $k=n$, the estimate is currently still unknown.
\end{remark}

\begin{proof} 
By the maximum principle, we have $u<0$ in $\Omega$. Let $\kappa_{\max}$ denote the maximum eigenvalue of the Hessian matrix $D^2u$, we consider the test function
\[Q(x)=\beta \log (-u)+\log \kappa_{\max} + \frac{a}{2}|Du|^2,\] where $a>0$ is a constant to be chosen later and it will turn out that $\beta>0$ can be arbitrary. Suppose that $Q$ attains its maximum at some point $x_0$. By choosing appropriate coordinates around $x_0$, we may assume that
\[\kappa_i=u_{ii}, \quad u_{ij}=\delta_{ij}u_{ii}, \quad \kappa_1 \geq \kappa_2 \geq \cdots \geq \kappa_n \quad \text{at $x_0$}.\] In case $\kappa_1$ has multiplicity $m>1$ i.e.
\[\kappa_1=\kappa_2 = \cdots = \kappa_m >\kappa_{m+1} \geq \cdots \geq \kappa_n,\] we may \footnote{Alternatively, one may also apply a standard perturbation argument; see e.g. \cite{Tu}.} apply a smooth approximation lemma \cite[Lemma 5]{BCD} of Brendle-Choi-Daskalopoulos to obtain that
\begin{align}
\delta_{\alpha\beta} \cdot (\kappa_{1})_{i}&=u_{\alpha\beta i}, \quad 1 \leq \alpha,\beta \leq m \\
(\kappa_1)_{ii}&\geq u_{11ii}+2\sum_{p>m}\frac{u_{1pi}^2}{\kappa_1-\kappa_p}
\end{align} in the viscosity sense.

Thus, at $x_0$ we have that
\begin{align}
0&=\beta \frac{u_i}{u} + \frac{u_{11i}}{u_{11}} + au_iu_{ii}, \label{Pogorelov 1st critical}\\
0&\geq \beta\frac{u_{ii}}{u}-\beta \frac{u_{i}^2}{u^2}+\frac{u_{11ii}}{u_{11}}-\frac{u_{11i}^2}{u_{11}^2}+a\left(u_{ii}^2+\sum_{k}u_{k}u_{kii}\right). \label{Pogorelov 2nd critical 1}
\end{align} Contracting \eqref{Pogorelov 2nd critical 1} with $F^{ii}$, we have that 
\begin{gather}\label{Pogorelov 2nd critical 2}
\begin{split}
0&\geq \frac{\beta f}{u}-\beta \sum F^{ii}\frac{u_{i}^2}{u^2}+\frac{1}{u_{11}}\sum F^{ii}u_{11ii}-\sum \frac{F^{ii}u_{11i}^2}{u_{11}^2}\\
&\quad +a\sum F^{ii}u_{ii}^2+a\sum_i \sum_k F^{ii}u_{k}u_{kii}.
\end{split}
\end{gather} 
Let us first estimate the fourth order term. By differentiating the equation
\[F(\kappa)=G(\lambda)=\frac{\sigma_{k}(\eta[u])}{\sigma_{l}(\eta[u])}=f(x,u,Du)\] twice, we have that
\begin{equation}
\sum F^{ii}u_{iik}=\sum G^{ii}\eta_{iik}=f_k=f_{x_k}+f_{u}u_{k}+f_{u_k}u_{kk} \label{Pogorelov differentiate once}
\end{equation} and 
\[G^{ij,rs}\eta_{ij1}\eta_{rs1}+G^{ii}\eta_{ii11}=f_{11} \geq -C(1+u_{11}+u_{11}^2)+f_{u_i}u_{11i}.\] As in section \ref{global curvature estimate proof}, by concavity of $G$ and \eqref{Pogorelov differentiate once}, we have
\[-G^{ij,rs}\eta_{ij1}\eta_{rs1} \geq -Cu_{11}^2.\] Hence,
\[F^{ii}u_{11ii}=G^{ii}\eta_{ii11}=f_{11}-G^{ij,rs}\eta_{ij1}\eta_{rs1}.\]
Our inequality \eqref{Pogorelov 2nd critical 2} then becomes
\begin{gather}\label{Pogorelov 2nd critical 3}
\begin{split}
0&\geq a\sum F^{ii}u_{ii}^2 +\frac{\beta f}{u} -\beta \sum F^{ii}\frac{u_{i}^2}{u^2} -\sum \frac{F^{ii}u_{11i}^2}{u_{11}^2}\\
&\quad +\left[a\sum_{i}\sum_{k} F^{ii}u_ku_{kii} + \sum_{i}\frac{\partial f}{\partial u_i}\frac{u_{11i}}{u_{11}}\right]-Cu_{11}-C.
\end{split}
\end{gather}
By the first order critical equation \eqref{Pogorelov 1st critical} and equality \eqref{Pogorelov differentiate once}, we have
\[a\sum_{i}\sum_{k} F^{ii}u_ku_{kii} + \sum_{i}\frac{\partial f}{\partial u_i}\frac{u_{11i}}{u_{11}} \geq -C\left(a+\frac{\beta}{(-u)}\right).\] For the third order terms, we again use the first order critical equation \eqref{Pogorelov 1st critical} to obtain that
\[\frac{u_{11i}^2}{u_{11}^2}=\left(\beta \frac{u_i}{u}+au_iu_{ii}\right)^2 \leq 2\beta^2\frac{u_{i}^2}{u^2}+2Ca^2u_{ii}^2\] for some $C>0$ depending on $\norm{u}_{C^1}$. Therefore, we have
\begin{gather}\label{Pogorelov 2nd critical 4}
\begin{split}
0&\geq \left(a-2Ca^2\right)\sum F^{ii}u_{ii}^2  +\frac{\beta f}{u} -\left(\beta+2\beta^2\right) \sum F^{ii}\frac{u_{i}^2}{u^2} \\
&\quad -Cu_{11}-C\left(a+\frac{\beta}{(-u)}\right)-C.
\end{split}
\end{gather}
Now, using
\[\sum F^{ii} \geq C(n,k,l,\inf f)\] and choosing $a>0$ small enough, we obtain that
\begin{gather}\label{Pogorelov 2nd critical 5}
\begin{split}
0&\geq \frac{a}{2}\sum F^{ii}u_{ii}^2 -Cu_{11}-C\left(1+a+\frac{\beta}{(-u)}+\frac{\beta+C\beta^2}{(-u)^2}\right)\sum F^{ii}.
\end{split}
\end{gather}

Finally, by proceeding exactly as in the proof of lemma \ref{key control}, we have
\[0 \geq \left(Ca\sum F^{ii}\right)u_{11}^2-\frac{C(a,\beta)}{(-u)^2}\sum F^{ii}\] and the desired estimate follows by dividing $\sum F^{ii}$.
\end{proof}

\section{The Purely Interior Estimate} \label{interior proof}
In this section, we derive purely interior $C^2$ estimates for equation \eqref{equation 2} when $0 \leq l<k<n$, and hence proving theorem \ref{interior C^2 estimates}. As reviewed in the Introduction, the $k-l \geq 2$ case has already been solved in \cite{Mei}. The new contribution here is the resolution for the $k-l=1$ (with $k<n$) case with a general right-hand side $f=f(x,u,Du)$ i.e. the following theorem.
\begin{theorem}
Let $1 \leq k < n$. Suppose $u \in C^4(\overline{B}_{1})$ is a $(\eta,k)$-convex solution of 
\begin{equation}
\frac{\sigma_k}{\sigma_{k-1}}(\eta[u])=f(x,u,Du). \label{interior equation}
\end{equation} Then there exists some $C>0$ depending on $n, k, \norm{u}_{C^1(B_1)}, \norm{f}_{C^2(B_1)}$ and $\inf_{B_1} f$ such that
\[\sup_{B_{1/2}}|D^2u| \leq C.\]
\end{theorem}
\begin{remark}
Our proof works for all $0 \leq l<k<n$.
\end{remark}
\begin{remark}
The estimate for the $(n,n-1)$ and $(n,n-2)$ quotients remains still unknown.
\end{remark}
\begin{proof}
For $x \in B_1$ and $\xi \in \bS^{n-1}$, we consider the standard test function
\[\tilde{Q}(x,\xi)=2 \log \rho(x) + \log \max\{ u_{\xi\xi}, N\} + \frac{a}{2}|Du|^2\] where $\rho(x)=1-|x|^2$ and $a>0$ is a constant to be determined later.

The maximum value of $\tilde{Q}$ in $\overline{B}_1 \times \bS^{n-1}$ will be attained at some interior point $x_0 \in B_1$, since $\rho=0$ on the boundary. Suppose also that the maximum direction is $\xi(x_0)=e_1$. Then we would have $u_{1i}(x_0)=0$ for $i \geq 2$. Indeed, let $\xi(t)=\frac{(1,t,0,\ldots,0)}{\sqrt{1+t^2}}$, we have
\[0=\frac{d}{dt}\tilde{Q}(x_0,\xi(t))\bigg|_{t=0}=2u_{12}(x_0).\] Choosing appropriate coordinates such that $D^2u(x_0)$ is diagonal and 
\[u_{11} \geq u_{22} \geq \cdots \geq u_{nn}.\] We shall also assume $u_{11} \geq N$ where $N>0$ is some possibly large number.

Now, the function
\[Q(x)=2 \log \rho(x) + \log u_{11} + \frac{a}{2}|Du|^2\] will also attain its maximum at $x_0$. Thus at this point, we have
\begin{align}
0&=2\frac{\rho_i}{\rho}+\frac{u_{11i}}{u_{11}}+au_iu_{ii}, \label{interior 1st critical} \\
0&\geq 2\frac{\rho_{ii}}{\rho}-2\frac{\rho_{i}^2}{\rho^2}+\frac{u_{11ii}}{u_{11}}-\frac{u_{11i}^2}{u_{11}^2}+a\left(u_{ii}^2+\sum_k u_k u_{kii}\right). \label{interior 2nd critical}
\end{align} Contracting \eqref{interior 2nd critical} with $F^{ii}$, we have 
\begin{gather} \label{interior 2nd critical 1}
\begin{split}
0&\geq -4\frac{\sum F^{ii}}{\rho}-8\frac{\sum F^{ii}x_{i}^2}{\rho^2} + \sum \frac{F^{ii}u_{11ii}}{u_{11}}-\sum \frac{F^{ii}u_{11i}^2}{u_{11}^2} \\
&\quad +a\sum F^{ii}u_{ii}^2 + a \sum_i\sum_k F^{ii}u_{iik}u_k+A\sum F^{ii}.
\end{split}
\end{gather} For the fourth order term, we differentiate equation \eqref{equation 2} twice to obtain that
\begin{equation}
F^{ii}u_{iik}=G^{ii}\eta_{iik}=f_k=f_{x_k}+f_uu_k+f_{u_k}u_{kk} \label{interior differentiate once}
\end{equation} and 
\[G^{ij,rs}\eta_{ij1}\eta_{rs1}+G^{ii}\eta_{ii11}=f_{11} \geq -C(1+u_{11}+u_{11}^2)+\sum \frac{\partial f}{\partial u_i}u_{11i}.\] As in section \ref{global curvature estimate proof}, by concavity of $G$ and \eqref{interior differentiate once}, we have 
\[-G^{ij,rs}\eta_{ij1}\eta_{rs1} \geq -Cu_{11}^2.\] Hence,
\[F^{ii}u_{11ii}=G^{ii}\eta_{ii11}=-G^{ij,rs}\eta_{ij1}\eta_{rs1}+f_{11}\] and \eqref{interior 2nd critical 1} becomes
\begin{gather} \label{interior 2nd critical 2}
\begin{split}
0&\geq  a \sum F^{ii}u_{ii}^2  -\sum \frac{F^{ii}u_{11i}^2}{u_{11}^2} -C\left(\frac{1}{\rho}+\frac{1}{\rho^2}\right)\sum F^{ii}\\
&\quad  + \left[a \sum_i\sum_k F^{ii}u_{iik}u_k +\sum \frac{\partial f}{\partial u_i}\frac{u_{11i}}{u_{11}}\right]- Cu_{11}-C.
\end{split}
\end{gather} By the first order critical equation \eqref{interior 1st critical} and \eqref{interior differentiate once}, we can estimate
\begin{align*}
&\ a \sum_i\sum_k F^{ii}u_{iik}u_k +\sum \frac{\partial f}{\partial u_i}\frac{u_{11i}}{u_{11}} \\
=&\ a \sum_{k} f_ku_k - \sum_{k} \frac{\partial f}{\partial u_k}\left(2\frac{\rho_k}{\rho}+au_ku_{kk}\right)\\
=&\ a \sum_{k} \left(f_{x_k}+f_{u}u_k+\frac{\partial f}{\partial u_k}u_{kk}\right) u_k - \sum_{k} \frac{\partial f}{\partial u_k}\left(2\frac{\rho_k}{\rho}+au_ku_{kk}\right)\\
=&\ a\sum_{k} (f_{x_k}+f_{u}u_{k})u_{k}-\sum_{k}\frac{\partial f}{\partial u_k}\left(-4\frac{x_k}{\rho}\right)\\
\geq &\ -C\left(a+\frac{1}{\rho}\right).
\end{align*} So we have
\begin{gather} \label{interior 2nd critical 3}
\begin{split}
0&\geq  a \sum F^{ii}u_{ii}^2  -\sum \frac{F^{ii}u_{11i}^2}{u_{11}^2} -C\left(\frac{1}{\rho}+\frac{1}{\rho^2}\right)\sum F^{ii}\\
&\quad  -C\left(a+\frac{1}{\rho}\right)- Cu_{11}-C.
\end{split}
\end{gather}
Again, by the first order critical equation \eqref{interior 1st critical}, we have
\[\frac{u_{11i}^2}{u_{11}^2}=\left(-\frac{4x_i}{\rho}+au_iu_{ii}\right)^2 \leq \frac{C}{\rho^2}+Ca^2u_{ii}^2\] for some $C>0$ depending on $\norm{u}_{C^1}$. Therefore, we have
\begin{gather} \label{interior 2nd critical 4}
\begin{split}
0&\geq  \left(a-Ca^2\right) \sum F^{ii}u_{ii}^2  -C\left(\frac{1}{\rho}+\frac{1}{\rho^2}\right)\sum F^{ii}\\
&\quad  -C\left(a+\frac{1}{\rho}\right)- Cu_{11}-C.
\end{split}
\end{gather} Finally, by using
\[\sum F^{ii} \geq C(n,k,l,\inf f)\] and choosing $a>0$ sufficiently small, we have
\begin{gather} \label{interior 2nd critical 5}
\begin{split}
0&\geq  \frac{a}{2} \sum F^{ii}u_{ii}^2 -C(1+a+u_{11})\sum F^{ii}- C\left(\frac{1}{\rho}+\frac{1}{\rho^2}\right)\sum F^{ii}.
\end{split}
\end{gather} By proceeding as in the proof of lemma \ref{key control}, we would have
\[0 \geq \left(Ca\sum F^{ii}\right)u_{11}^2-C\left(\frac{1}{\rho}+\frac{1}{\rho^2}\right)\sum F^{ii}\] and the desired estimate follows by dividing $\sum F^{ii}$.

\end{proof}

\appendix

\section{A note on interior curvature estimates}
In this appendix, we prove an interior curvature estimate for the following curvature quotient equation.
\begin{theorem}\label{interior curvature estimates}
Let $1 \leq k \leq n$.
Suppose $\Sigma=(x,u(x))$ is a $k$-convex graph over $B_r \subseteq \bR^n$ and it is a solution of the following curvature quotient equation
\[\frac{\sigma_k}{\sigma_{k-1}}(\kappa_1(x),\ldots,\kappa_n(x))=f(X)>0, \quad X \in B_r \times \bR.\] Then there exists some $C>0$ depending on $n,k,r,\norm{\Sigma}_{C^1(B_r)}, \inf_{B_r} f$ and $\norm{f}_{C^2(B_r)}$ such that 
\[\sup_{x \in B_{r/2}}|\kappa_{\max}(x)| \leq C.\]
\end{theorem}

\begin{remark}
This result has long been known by experts in the field, at least since the work of Sheng-Urbas-Wang \cite{Sheng-Urbas-Wang}; this is why we are stating the result in an appendix rather than a formal section. In fact, we were told about the proof by Professor Siyuan Lu and Professor Guohuan Qiu in private communications. Moreover, in Lu's recent work \cite{Lu-2}, it is pointed out that the result was obtained by Professor Pengfei Guan and Professor Xiangwen Zhang in their unpublished note. 

The reasons for still including a proof of theorem \ref{interior curvature estimates} are (1) the proof of theorem \ref{interior curvature estimates 2} would be exactly the same (which is our main concern in this note); (2) there is no written proof for the result in the literature (not that we know of) and we thought we might as well provide one while getting theorem \ref{interior curvature estimates 2}.

\end{remark}
\begin{proof}
For simplicity of notation, we work on the case $r=1$; the arguments can be readily carried over for general $r>0$. The standard basis of $\bR^{n+1}$ will be denoted by $\{E_1,\ldots,E_{n+1}\}$ and the usual Euclidean inner product is denoted by $\langle \cdot , \cdot \rangle$. Let $X(x)$ be the position vector of $\Sigma$ and the outer unit normal is given by
\[\nu=\frac{(-Du,1)}{\sqrt{1+|Du|^2}};\] where we are going to make use of 
\[\nu^{n+1}=\langle \nu , E_{n+1}\rangle = \frac{1}{\sqrt{1+|Du|^2}}.\] For any unit tangential vector $\vartheta$ on $\Sigma$, we consider the auxiliary function in $B_1$:
\[Q(X(x),\vartheta)=\log \rho(X) + \log h_{\vartheta\vartheta}-\log (\nu^{n+1}-a),\] where $\rho(X)=1-|X|_{\bR^{n+1}}^{2}+\langle X, E_{n+1}\rangle^2=1-|x|_{\bR^{n}}^2$ and $0<a<\frac{1}{2}\inf \nu^{n+1}$ is some number depending on $\sup |Du|$. Suppose $Q$ attains its maximum at some interior $x_0 \in B_1$ in the direction of $\vartheta(x_0)$, which we may take to be $e_1$. Forming a local orthonormal frame $\{e_1,\ldots,e_n\}$ by rotating $\{e_2,\ldots,e_n\}$, we may assume that $h_{ij}(x_0)=\kappa_i\delta_{ij}$ and $\kappa_1 \geq \cdots \geq \kappa_n$. Denote the covariant differentiation in the direction $e_i$ by the symbol $\nabla_i$, we have at $x_0$, that
\begin{align}
0&=\frac{\rho_i}{\rho}+\frac{h_{11i}}{h_{11}}-\frac{\nabla_i\nu^{n+1}}{\nu^{n+1}-a}, \label{1st critical}\\
0&\geq \frac{\rho_{ii}}{\rho}-\frac{\rho_{i}^2}{\rho^2}+\frac{h_{11ii}}{h_{11}}-\frac{h_{11i}^2}{h_{11}^2}-\frac{\nabla_{ii}\nu^{n+1}}{\nu^{n+1}-a}+\left(\frac{\nabla_{i}\nu^{n+1}}{\nu^{n+1}-a}\right)^2. \label{2nd critical 1}
\end{align}

By the commutator formula i.e.
\begin{equation}
h_{11ii}=h_{ii11}-h_{ii}^2h_{11}+h_{ii}h_{11}^2,\label{commutator formula}
\end{equation} and contracting \eqref{2nd critical 1} with $F=\sigma_k/\sigma_{k-1}$, we have that
\begin{gather} \label{2nd critical 2}
\begin{split}
0&\geq \sum \frac{F^{ii}\rho_{ii}}{\rho}-\sum \frac{F^{ii}\rho_{i}^2}{\rho^2}+\sum \frac{F^{ii}h_{ii11}}{\kappa_1} - \sum F^{ii}\kappa_{i}^2 + F \kappa_1\\
&\quad -\sum F^{ii}\frac{h_{11i}^2}{\kappa_{1}^2} - \sum \frac{F^{ii}\nabla_{ii}\nu^{n+1}}{\nu^{n+1}-a}+\sum F^{ii}\left(\frac{\nabla_i\nu^{n+1}}{\nu^{n+1}-a}\right)^2,
\end{split}
\end{gather} where we have used lemma \ref{properties of sigma-k} to get 
\[\sum F^{ii}h_{ii}=F.\]
To proceed, we need evaluate some of the terms in \eqref{2nd critical 2}. By differentiating the curvature equation twice, we have that
\begin{equation}
\sum_{i=1}^{n} F^{ii}h_{ii1}=\nabla_1f(X)
\end{equation} and
\begin{align*}
F^{ii}h_{ii11}+F^{ij,rs}h_{ij1}h_{rs1}&=\nabla_{11}f(X) \geq -C\kappa_1.
\end{align*} Note that, by lemma \ref{concavity}, we have
\[-F^{ij,rs}h_{ij1}h_{rs1}=-F^{pp,qq}h_{pp1}h_{qq1}+F^{pp,qq}h_{pq1}^2 \geq 0.\]
Also, we have that \cite[Lemma 2.2]{Sheng-Urbas-Wang}
\[F^{ij}\nabla_{ij}\nu^{n+1}+ F^{ij}h_{ik}h_{jk}\nu^{n+1}=-\langle \nabla f, e_{n+1}\rangle.\] Thus, it follows from \eqref{2nd critical 2} that
\begin{gather} \label{2nd critical 3}
\begin{split}
0& \geq \sum \frac{F^{ii}\rho_{ii}}{\rho}-\sum \frac{F^{ii}\rho_{i}^2}{\rho^2}-\sum F^{ii}\frac{h_{11i}^2}{\kappa_{1}^2}+\sum F^{ii}\left(\frac{\nabla_i\nu^{n+1}}{\nu^{n+1}-a}\right)^2\\
&\quad +F\kappa_1 + \frac{a}{\nu^{n+1}-a}\sum F^{ii}\kappa_{i}^2-C.
\end{split}
\end{gather} Next, for the third order term, we use the first order critical equation \eqref{1st critical} and apply the Cauchy's inequality:
\[\sum F^{ii}\frac{h_{11i}^2}{\kappa_{1}^2} \leq \sum F^{ii}\left(\frac{\nabla_i \nu^{n+1}}{\nu^{n+1}-a}-\frac{\rho_i}{\rho}\right)^2 \leq (1+\epsilon)\sum F^{ii} \left(\frac{\nabla_i\nu^{n+1}}{\nu^{n+1}-a}\right)^2+\left(1+\frac{1}{\epsilon}\right)\sum F^{ii}\frac{\rho_{i}^2}{\rho^2}\] for some $\epsilon>0$ to be determined later. Substituting this into \eqref{2nd critical 3}, we obtain that
\begin{gather}\label{2nd critical 4}
\begin{split}
0&\geq \sum \frac{F^{ii}\rho_{ii}}{\rho}-\left(2+\frac{1}{\epsilon}\right)\sum \frac{F^{ii}\rho_{i}^2}{\rho^2}\\
&\quad +\frac{a}{\nu^{n+1}-a}\sum F^{ii}\kappa_{i}^2-\epsilon\sum F^{ii}\left(\frac{\nabla_i\nu^{n+1}}{\nu^{n+1}-a}\right)^2\\
&\quad + F\kappa_1-C.
\end{split}
\end{gather} For the second line in \eqref{2nd critical 4}, we use
\[\nabla_{i}\nu^{n+1}=h_{ii}\langle e_i, E_{n+1}\rangle \quad \text{and} \quad \frac{1}{\nu^{n+1}-a} \leq \frac{1}{a} \] to get
\begin{align*}
&\ \frac{a}{\nu^{n+1}-a}\sum F^{ii}\kappa_{i}^2-\epsilon\sum F^{ii}\left(\frac{\nabla_i\nu^{n+1}}{\nu^{n+1}-a}\right)^2\\
\geq &\ \frac{1}{\nu^{n+1}-a}\sum F^{ii} \left[a\kappa_{i}^2-\frac{\epsilon}{a}\cdot C\kappa_{i}^2\right]\\
\geq &\ 0
\end{align*} by choosing $\epsilon$ sufficiently small. Hence, we are left with
\begin{gather}\label{2nd critical 5}
\begin{split}
0&\geq \sum \frac{F^{ii}\rho_{ii}}{\rho}-\left(2+\frac{1}{\epsilon}\right)\sum \frac{F^{ii}\rho_{i}^2}{\rho^2}\\
&\quad + F\kappa_1-C
\end{split}
\end{gather} and it remains to deal with the terms involving $\rho$. By a direct computation, we have
\begin{align*}
\rho_i&=-2\langle X,e_i\rangle+2\langle X, E_{n+1}\rangle \langle e_i,E_{n+1}\rangle,\\
\rho_{ii}&=-2+2h_{ii}\langle X,\nu \rangle + 2 \langle e_i, E_{n+1}\rangle-2h_{ii}\langle X, E_{n+1}\rangle\nu^{n+1}.
\end{align*} In our situation, it is enough to get the following rough estimate
\[\sum \frac{F^{ii}\rho_{ii}}{\rho}-\left(2+\frac{1}{\epsilon}\right)\sum \frac{F^{ii}\rho_{i}^2}{\rho^2} \geq -C\frac{\sum F^{ii}}{\rho^2},\] and then, since 
\[\sum F^{ii} \leq n-k+1 \quad \text{for $F=\frac{\sigma_k}{\sigma_{k-1}}$},\] the inequality \eqref{2nd critical 5} becomes
\[0 \geq F\kappa_1 - \frac{C}{\rho^2} - C\] and the desired estimate follows.

\end{proof}

\begin{remark} \label{key observation}
It is crucial that we have got an extra positive term $F\kappa_1$ from \eqref{commutator formula}, which would be absent if we were dealing with the Hessian quotient equation i.e. 
\[\frac{\sigma_k}{\sigma_{k-1}}(D^2u)=f(x,u).\]
\end{remark}

\subsection*{Statements and Declarations}
The author has no competing interests to declare that are relevant to the content of this article.

\bibliography{refs}
\end{document}